\documentclass[11pt, reqno]{amsart}

\newtheorem{theorem}{Theorem}[subsection]

\newtheorem{Corollary}[theorem]{Corollary}
\newtheorem{proposition}[theorem]{Proposition}

\theoremstyle{definition}
\newtheorem{definition}[theorem]{Definition}

\newlength{\margins}
\usepackage[top=\margins,bottom=\margins,left=\margins,right=\margins]{geometry}
\usepackage{amsmath}
\numberwithin{equation}{section}

\usepackage{ amscd, mathrsfs, dsfont}
\usepackage{amsmath,amssymb,amsthm,mathscinet}
\usepackage{enumerate}
\usepackage{easybmat,graphics}
\usepackage{etex, tikz, epic}
\usepackage{hyperref}
\hypersetup{colorlinks=true,citecolor=black,linkcolor=black}
\usepackage{etex}
\usepackage[all]{xy}
\usepackage[mathscr]{euscript}
\makeatletter
\newcommand\xleftrightarrow[2][]{%
  \ext@arrow 9999{\longleftrightarrowfill@}{#1}{#2}}
\newcommand\longleftrightarrowfill@{%
  \arrowfill@\leftarrow\relbar\rightarrow}
\makeatother

\newtheorem{Remark}[theorem]{Remark}

\def\depth{\operatorname{depth}}

\def\det{\operatorname{det}}

\def\WD{\operatorname{ WD}}

\def\ind{\operatorname{ ind\, }}
\def\Char{\operatorname{char}}
\def\cond{\operatorname{ cond }}

\def\Gal{\mathop{\rm Gal}\nolimits}
\def\GSp{\mathop{\rm GSp}\nolimits}

\def\GSpin{\mathop{\rm GSpin}\nolimits}

\def\SO{\mathop{\rm SO}\nolimits}
\def\GL{\mathop{\rm GL}\nolimits}
\def\Sym{\mathop{\rm Sym}\nolimits}

\def \Del{\mathop{\rm Del}\nolimits}

\def \Kaz{\mathop{\rm Kaz}\nolimits}

\numberwithin{equation}{section}

\numberwithin{equation}{section}

\newcommand{\calo}{\mathcal{O}}
\newcommand{\calp}{\mathfrak{p}}
\newcommand{\cala}{\mathfrak{a}}

\newcommand{\fH}{\mathscr{H}}
\newcommand{\fn}{\mathfrak{n}}

\newcommand{\bfm}{{\bf M}}
\newcommand{\bfb}{{\bf B}}
\newcommand{\bfg}{{\bf G}}

\newcommand{\bfp}{{\bf P}}
\newcommand{\bfn}{{\bf N}}

\newcommand{\bft}{{\bf T}}

\newcommand{\bfu}{{\bf U}}

\newcommand{\lu}{{\bf u}}

\newcommand{\Z}{\mathbb{Z}}
\newcommand{\C}{\mathbb{C}}
\newcommand{\bG}{\mathbb{G}}
\newcommand{\Q}{\mathbb{Q}}
\newcommand{\R}{\mathbb{R}}

\newcommand{\F}{\mathbb{F}}
\newcommand{\tw}{\tilde{w}}

\begin{document}

\title[Twisted exterior and symmetric square $\gamma$-factors]{On twisted exterior and symmetric square $\gamma$-factors}


\author{Radhika Ganapathy}

\author{Luis Lomel\'i}

%\keywords{Automorphic $L$-functions, local factors, LS method, Plancherel measures}

\begin{abstract}
We establish the existence and uniqueness of twisted exterior and symmetric square $\gamma$-factors in positive characteristic by studying the Siegel Levi case of generalized spinor groups. The corresponding theory in characteristic zero is due to Shahidi. In addition, in characteristic $p$ we prove that these twisted local factors are compatible with the local Langlands correspondence. As a consequence, still in characteristic $p$, we obtain a proof of the stability property of $\gamma$-factors under twists by highly ramified characters. Next we use the results on the compatibility of the Langlands-Shahidi local coefficients with the Deligne-Kazhdan theory over close local fields to show that the twisted symmetric and exterior square $\gamma$-factors, $L$-functions and $\varepsilon$-factors are preserved over close local fields.  Furthermore, we obtain a formula for Plancherel measures in terms of local factors and we also show that they also preserved over close local fields.
\end{abstract}

\maketitle

\section{Introduction}
We make a thorough study of the induction step of the Langlands-Shahidi method for the generalized spinor groups in positive characteristic, namely the case of a Siegel parabolic. Based on the results of Henniart-Lomel\'i \cite{HL11} and the study of the Langlands-Shahidi method over function fields begun in \cite{Lom09}, we prove the existence and uniqueness of twisted exterior and symmetric square $\gamma$-factors, $L$-functions and root numbers in positive characteristic. The theory in characteristic zero is entirely due to Shahidi \cite{Sha90, Sha97}.  We then use the local-to-global technique of \cite{HL11} to prove that our local twisted factors in positive characteristic are compatible with the local Langlands correspondence for $\GL_n$. Combining this with the results of \cite{DH81}, we obtain an imporant stability property of the twisted symmetric and exterior square $\gamma$-factors under twists by highly ramified characters.  Next we use the results of \cite{Gan13} on the compatibility of the Langlands-Shahidi local coefficients with the Deligne-Kazhdan theory to deduce that these twisted local factors are compatible with Kazhdan isomorphism \cite{Kaz86, How85, Lem01, Gan13} over sufficiently close local fields. The compatibility of the Artin factors with the Deligne isomorphism over close local fields is due to Deligne \cite{Del84}. We finally express the Plancherel measures as a product of $\gamma$-factors and deduce that they are also preserved over sufficiently close local fields.

In order to be more precise, we introduce classes $\mathscr{L}$ and $\mathscr{G}$ where twisted local factors are defined. We let $\mathscr{L}$ be the class whose objects are quadruples $(F,\pi,\eta,\psi)$ consisting of: a non-archimedean local field $F$, a smooth irreducible representation $\pi$ of ${\rm GL}_n(F)$; a multiplicative character $\eta$ of $F^\times$; and, a non-trivial additive character $\psi$ of $F$. Similarly, we let $\mathscr{G}$ be the class of quadruples $(F,\sigma,\eta,\psi)$, where $\sigma$ is now a complex $n$-dimensional ${\rm Frob}$-semisimple Weil-Deligne representation. Furthermore, given a prime $p$, we denote by $\mathscr{L}(p)$ and $\mathscr{G}(p)$ the subclasses of $\mathscr{L}$ and $\mathscr{G}$ having ${\rm char}(F) = p$.

We let ${\rm LLC}: \mathscr{L} \rightarrow \mathscr{G}$ be the morphism that to each $(F,\pi,\eta,\psi) \in \mathscr{L}$ assigns a quadruple $(F,\sigma,\eta,\psi) \in \mathscr{G}$ with $\sigma$ obtained from $\pi$ via the local Langlands correspondence for ${\rm GL}_n(F)$ \cite{HT01, Hen00, LRS93} and we view $\eta$ as a character of $\mathcal{W}_F$ via local class field theory. Let $r_0$ be either ${\rm Sym}^2$ or $\wedge^2$ and let $\gamma(s,\pi,r_0 \boxtimes \eta,\psi)$ denote the twisted exterior or symmetric square $\gamma$-factor. On the Galois side, $\gamma(s,r_0(\sigma)\otimes \eta,\psi)$ will be the Galois $\gamma$-factor defined by Langlands and Deligne \cite{t79}. Restricting ourselves to positive characteristic, in Theorem~4.3 we prove that the morphism
\begin{equation*}
\begin{CD}
   \mathscr{L}(p) @>{\rm LLC}>> 	\mathscr{G}(p)
\end{CD}
\end{equation*}
preserves twisted exterior and symmetric square local factors. In particular, we have
\begin{equation} \label{LLCcompatibility}
   \gamma(s,\pi,r_0 \boxtimes \eta,\psi) = \gamma(s,r_0(\sigma)\otimes\eta,\psi).
\end{equation}
We remark that, for non-archimedean local fields of characteristic 0, Henniart proved that the above equality holds up to a root of unity  in \cite{Hen2010}. It is the aim of current work of Cogdell, Shahidi and Tsai \cite{CST}, still in characteristic 0, to prove the Equality~\eqref{LLCcompatibility}. 

An imporant consequence of Equation \eqref{LLCcompatibility} is  a stability property of $\gamma$-factors under twists by highly ramified characters; indeed, let $(F,\pi_i,\eta,\psi) \in \mathscr{L}(p)$, $i = 1$, $2$, be such that their central characters $\omega_{\pi_i}$ are equal and $\eta$ is sufficiently ramified. Then, Proposition \ref{Stability} states that
\begin{equation} \label{stabilityintro}
   \gamma(s,\pi_1,r_0 \boxtimes \eta,\psi) = \gamma(s,\pi_2,r_0 \boxtimes \eta,\psi).
\end{equation}
We remark that, in the case of non-archimedean local fields of characteristic 0, Cogdell, Shahidi and Tsai work directly with the Langlands-Shahidi local coefficient and its connection to Bessel models in \cite{CST}. They directly prove the above stability property of $\gamma$-factors, and from Equation~\eqref{stabilityintro} they are able to establish Equation~\eqref{LLCcompatibility}, contrary to our proof in the characteristic $p$ case.

Having a solid theory of twisted local factors over non-archimedean local fields in hand, now including the case of characteristic $p$, allows us incorporate the results of \cite{Gan12, Gan13} on the compatibility of the Langlands-Shahidi local coefficients with the Deligne-Kazhdan correspondence \cite{Del84, Kaz86} on the representation theory over close local fields. Recall that two non-archimedean local fields $F$ and $F'$ are \emph{m-close} if $\calo_F/\calp_F^m \cong \calo_{F'}/\calp_{F'}^m$. For example, the fields $\F_p(\!(t)\!)$ and $\Q_p(p^{1/m})$ are $m$-close. More generally, a local field of positive characteristic can be approximated with non-archimedean local fields of characteristic 0. In \cite{Del84}, Deligne proved that when the fields $F$ and $F'$ are $m$-close, the absolute Galois groups of $F$ and $F'$ modulo their respective $m$-th higher ramification subgroups with upper numbering become isomorphic, thereby giving a bijection 
\begin{equation}\label{Delbijection}\{\sigma: \WD_F \rightarrow \GL_n(\C)| \depth(\sigma) < m\} \leftrightarrow \{\sigma': \WD_{F'} \rightarrow \GL_n(\C)| \depth(\sigma') < m\}.
\end{equation}
Now, let $\mathscr{G}^m$ be the subclass of $\mathscr{G}$ consisting of quadruples $(F,\sigma,\eta,\psi)$ with $\depth(\sigma)< m$ and $\depth(\eta) < m$.   
We say that that the quadruple $(F, \sigma, \eta, \psi) \in \mathscr{G}^m$  is $\Del_m$-associated to $(F', \sigma', \eta', \psi')$ if: $F$ and $F'$ are $m$-close, $\sigma \leftrightarrow \sigma'$ and $\eta \leftrightarrow \eta'$ via Equation \eqref{Delbijection}, and $\psi$ is compatible with $\psi'$ (see Definition \ref{DelAsso}). A consequence of Proposition 3.7.1 of \cite{Del84} is that for each pair of quadruples 
\begin{equation*}
(F, \sigma, \eta, \psi) \xleftrightarrow{\Del_m\text{-associated}} ( F', \sigma', \eta', \psi')
\end{equation*}
one has
\begin{equation}\label{Introgamma}
\gamma(s,(r_0 \circ \sigma) \otimes \eta,\psi)  = \gamma(s,(r_0 \circ \sigma') \otimes \eta',\psi').
\end{equation}
The above equality also holds true for the Artin $L$-functions and $\epsilon$-factors. 

For split reductive groups $\bfg/\Z$, the Hecke algebra isomorphism over close local fields due to Kazhdan \cite{Kaz86} (its variant can be found in \cite{How85, Lem01} for $\GL_n$ and \cite{Gan13} for any split reductive group) gives a bijection between
\begin{align}\label{Kazbijection}
\{&\text{Irreducible admissible repns. } \pi \text{ of } \bfg(F)\;|\; \depth(\pi) < m\} \\\nonumber
&\overset{\kappa_m}\longleftrightarrow \{\text{Irreducible admissible repns. } \pi' \text{ of } \bfg(F')\;| \;\depth(\pi') < m\}
\end{align}
when the fields $F$ and $F'$ are $(m+1)$-close. In \cite{Gan13}, it was shown that when the fields $F$ and $F'$ are sufficiently close, the Langlands-Shahidi local coefficients are preserved for generic representations of Levi subgroups that correspond via Equation \eqref{Kazbijection}. We apply this result to deduce the analogue of Equation \eqref{Introgamma} on the analytic side for our twisted $\gamma$-factors (and also $L$-functions and root numbers).

More precisely, given $m \geq 1$, let $\mathscr{L}^m$ be the subclass of $\mathscr{L}$ whose objects consist of quadruples $(F,\pi,\eta,\psi)$ with ${\rm depth}(\pi) < m$ and ${\rm depth}(\eta)< m$. For $l \geq m+1$, we say that the quadruple $(F,\pi,\eta,\psi) \in \mathscr{L}^m$ is $\Kaz_l$-associated to $(F',\pi',\eta',\psi')$ if: $F$ and $F'$ are $l$-close, $\pi \leftrightarrow \pi'$ and $\eta \leftrightarrow \eta'$ via $\kappa_l$ of Equation \ref{Kazbijection}, and $\psi$ and $\psi'$ are compatible (See Definition \ref{KazAsso}). We deduce that there exists $l = l(n,m) \geq m+1$, depending only on $n$ and $m$, such that for each pair of quadruples 
\begin{equation}\label{KazIntroAsso}
(F, \pi, \eta, \psi) \xleftrightarrow{\Kaz_l\text{-associated}} ( F', \pi', \eta', \psi')
\end{equation}
one has
\begin{align*}
\gamma(s, \pi, r_0 \boxtimes \eta, \psi) &= \gamma(s, \pi', r_0 \boxtimes \eta', \psi').
\end{align*} 
In addition to the $\gamma$-factors, the twisted symmetric and exterior square $L$-functions and root numbers are also preserved by this transfer (see Proposition~\ref{Kazhdantwisted}).

Plancherel measures play an important role in the recent results on the local Langlands correspondence for ${\rm GSp}_4$ \cite{GT11, Gan13}. With respect to Plancherel measures associated to representations of Siegel Levi subgroups of spinor groups, they can be expressed in terms of $\gamma$-factors. This we do regardless of characteristic in Proposition~\ref{Plancherelgamma}, where we follow Shahidi \cite{Sha90}, who completely covered the characteristic zero case. We then deduce that for a pair of quadruples as in Equation \eqref{KazIntroAsso}, the Plancherel measures are also equal (see Corollary~\ref{Planchereltransfer}).

\subsection*{Acknowledgments} First of all, we are indebted to G. Henniart for his encouragement on pursuing this line of research. Many of the results of this article are possible thanks to his work and the Henniart-Lomel\'i collaboration.  The results in this article on the transfer of the twisted local factors over close local fields are consequences of the first author's Ph.D thesis project; she would like express her most sincere gratitude to her advisor Jiu-Kang Yu for suggesting her thesis project and for the numerous valuable discussions. We would like to thank F. Shahidi, to whom the theory of twisted exterior and symmetric square local factors in characteristic zero is due. In addition, we thank M. Asgari for taking the time to answer basic questions concerning generalized spinor groups. We also profited from fruitful mathematical discussions with Yeansu Kim, A. Roche, R. Schmidt and Sandeep Varma. The first author would like to thank the Institute for Advanced Study for providing a great working environment during the academic year 2012-2013 where some parts of this article were written.  The second author would like to thank the University of Oklahoma, in particular the Algebra and Representation Theory group, for providing a good professional environment while work on this paper was made.

\section{Generalized spinor groups}

\subsection{Special orthogonal groups}\label{qsSO} Given a commutative ring $R$ with identity, let $q$ be the quadratic form on $V = R^{2n}$ defined by
\begin{equation*}
   q(x) = \sum_{i=1}^{n} x_i x_{2n+1-i}, \ x \in R^{2n}.
\end{equation*}
Then ${\rm O}_{2n}(R) = \left\{ g \in {\rm GL}_{2n}(R) \, \vert \, q(gx) = q(x), \forall x \in R^{2n} \right\}$. Let $D_q: {\rm O}_{2n}(R) \rightarrow \mathbb{Z}_2(R)$ be the Dickson invariant (see \cite{knus,kmrt}). Then, ${\rm SO}_{2n}(R)$ is defined to be the kernel of $D_q$.

In contrast to even special orthogonal groups, which require special attention in characteristic two, odd special orthogonal groups ${\rm SO}_{2n+1}$ are defined via the usual determinant. This gives a smooth connected group scheme regardless of characteristic. We use the quadratic form

\begin{equation*}
   q(x) = \sum_{i=1}^n x_i x_{2n+2-i} + x_{n+1}^2, \ x \in R^{2n+1},
\end{equation*}
just as in sections~3.2--3.3 of \cite{Lom09}.

\subsection{Generalized spinor groups and their structure theory:} With a quadratic space $(V,q)$ as above, the generalized spinor group $\GSpin_m, m=2n+1$ or $2n$ can be realized as the central extension of $\SO(q)$ by $\mathbb{G}_m$ (Refer Appendix C of \cite{conrad} for details). This holds true regardless of characteristic. Note that ${\rm GSpin}_1\simeq {\rm GL}_1$.

Let us recall the structure theory of split $\GSpin$ groups from \cite{AS06}.  From now on, let $\bfg = \GSpin_m,\; m=2n+1$ or $2n,$ be the split connected reductive group over $\Z$.  Let 
\begin{align*}
X &:= \Z e_0 \oplus \Z e_1 \oplus \cdots \oplus\Z e_n,\\
 X^\vee &:= \Z e_0^* \oplus \Z e_1^* \oplus \cdots \oplus \Z e_n^*,
\end{align*}
and let $\langle\;, \; \rangle$ denote the standard $\Z$-pairing on $X \times X^\vee$. Then the root datum of $\GSpin_m$ can be described as $(X, \Phi, X^\vee, \Phi^\vee)$, with $\Phi$ and $\Phi^\vee$ generated, respectively, by 
\begin{align*}
\Delta&=\{\alpha_1= e_1-e_2, \ldots, \alpha_{n-1}= e_{n-1} -e_n, \alpha_n= e_n\}, \\
\Delta^\vee &=\{ \alpha_1^{\vee}= e_1^*-e_2^*, \ldots, \alpha_{n-1}^\vee =e_{n-1}^* -e_n^*, \alpha_n^\vee= 2e_n^* - e_0^*\},
\end{align*}
when $m = 2n+1$ and by
\begin{align*}
\Delta&=\{\alpha_1= e_1-e_2, \ldots, \alpha_{n-1}= e_{n-1} -e_n, \alpha_n= e_{n-1}+e_n\}, \\
\Delta^\vee &=\{ \alpha_1^{\vee}= e_1^*-e_2^*, \ldots, \alpha_{n-1}^\vee= e_{n-1}^* -e_n^*, \alpha_n^\vee= e_{n-1}^* + e_n^* - e_0^*\},
\end{align*}
when $m=2n$.  We fix a Chevalley basis $\{\lu_{\alpha}:\bG_a \rightarrow \bfu_\alpha|\alpha \in \Phi\}$ for $\bfg$, where $\bfu_\alpha$ the root subgroup associated to $\alpha \in \Phi$. Let $\bfb = \bft\bfu$ denote the Borel subgroup of $\bfg$ that determines the  positive roots $\Phi^+$ and $\Delta$ described above. Let $W$ denote the Weyl group of $\bfg$ and let $w_{l,\Delta}$ denote the longest element of $W$.  For each $w \in W$ we fix a reduced expression for $w$ and we always choose its representative $\tw$ using the Chevalley basis fixed above. Note that this representative is independent of the choice of reduced expression of $w$. 

Let $\theta: = \Delta \backslash \{\alpha_n\}$. Then $\theta$ determines the standard Siegel Levi subgroup $\bfm \cong \GL_n \times \GL_1$ of $\bfg$. Let $W_\theta$ denote the Weyl group of $\bfm$ and let $w_{l, \theta}$ denote the longest element of $W_\theta$. Set $w_0 = w_{l, \Delta}w_{l, \theta}$. Note that $w_0 \bfm w_0^{-1} = \bfm$, that is, the Siegel Levi subgroup is self-associate in $\bfg$. Let $\bfp = \bfm\bfn$ be the standard parabolic subgroup of $\bfg$ with $\bfn \subset \bfu$. Let $X^*(\bfm)$ denote the set of rational characters of $\bfm$. Let $\cala^*: = X^*(\bfm) \otimes_\Z \R$, $\cala$ its dual, and $\cala_{ \C}^*$ its complexification. Let $\rho$ denote the half-sum of the roots that generate $\bfn$ and let $\tilde{\alpha}_n = \langle \rho, \alpha_n^\vee \rangle^{-1} \rho$.

Consider the adjoint action $r$ of $^L\bfm$, the $L$-group of $\bfm$, on the Lie algebra $^L\fn$, the Lie algebra of the $L$-group of $\bfn$. By Proposition 5.6 of \cite{Asg02}, we know that
\begin{align} \label{adjointr}
r &=
  \begin{cases}
   \Sym^2 \boxtimes \, {\rm sim}^{-1}       & \text{if } \bfg = \GSpin_{2n+1} \\
   \wedge^2 \boxtimes {\rm sim}^{-1}       & \text{if } \bfg = \GSpin_{2n}
  \end{cases}.
\end{align}
The representations of Siegel Levi subgroups of split $\GSpin$ groups give rise to the theory of twisted symmetric and exterior square $\gamma$-factors to be studied in this article. 

\subsection{} Given a non-archimedean local field $F$, let $\mathcal{O}_F$ be its ring of integers, $\mathfrak{p}_F$ its maximal ideal, $q_F$ the cardinality of its residue field, and $\varpi_F$ a uniformizer. When working over a fixed $F$, we sometimes drop the subscripts when there is no possibility of confusion. Given a global field $k$ and a non-archimedean valuation $v$ of $k$, we write $\mathcal{O}_v$ for the ring of integers of $k_v$; and similarly for $\mathfrak{p}_v$, $q_v$ and $\varpi_v$. Given a representation $\tau$, we let $\tau^{\vee}$ denote its contragredient representation.
For an algebraic group ${\bf H}$ defined over $\Z$, we let $H = {\bf H}(F)$. We then have $G, B, T, U$ and so on.

When the group is ${\bf G} = {\rm GL}_n$, we write ${\bf B}_n = {\bf T}_n{\bf U}_n$ for the Borel subgroup of upper triangular matrices with maximal torus ${\bf T}_n$ and unipotent radical ${\bf U}_n$. The standard representation of ${\rm GL}_n(\mathbb{C})$ is denoted by $\rho_n$.

\subsection{Haar measures} \label{Haar} Normalization of Haar measures is done in such a way that the theory matches all the way to the basic semisimple rank 1 cases. In particular, we have compatibility with Tate's thesis. For this, we fix a non-archimedean local field $F$ and a non-trivial character $\psi$ of $F$. We let $\mu_\psi$ be the self dual Haar measure corresponding to $\psi$. That is, $\mu_\psi$ is chosen so that Fourier inversion holds (see for example Equation~(1.1) of \cite{Lom12b}). The measure on the set of rational points $N$ of the unipotent radical of the Siegel parabolic subgroup $\bf P$ is obtained by taking a product of measures $\mu_\psi$ on each of the corresponding one parameter subgroups $U_\alpha$. Let $I$ be the standard Iwahori subgroup of $G$ (see \S~4 of \cite{Gan13}). We then have
\begin{equation*}
   {\rm vol}(N \cap I, dn) = \mu_\psi(\mathcal{O}_F)^{\frac{n(n\pm1)}{2}},
\end{equation*}
where we use $+$ for odd spinor groups and we use $-$ for even spinor groups.

\subsection{Local coefficients}\label{LCPM} Let us briefly recall the theory of local coefficients and Plancherel measures from \cite{Sha81, Sha90} in the context of the Siegel Levi subgroups described above. 
Fix a non-trivial additive character $\psi$ of $F$. Let $\chi$ be the generic character of $U$ defined by
\[\chi = \displaystyle{\prod_{\alpha \in \Delta}} \psi \circ \lu_{\alpha}^{-1}.\]
Note that since the representatives of Weyl group elements are chosen using the same Chevalley basis, the character $\chi$ is compatible with $\tw_0$, that is, 
\[\chi(u) = \chi(\tw_0u\tw_0^{-1}) \;\forall\; u \in U \cap M.\]
Let $\chi_M$ denote the restriction of $\chi$ to $U \cap M$.   Let $\pi$ be an irreducible, admissible, generic representation of $\GL_n(F)$, $\eta$ a character of $F^\times$ and consider $\pi \boxtimes \eta^{-1}$ as a representation of $M$.  Since the center of $\bfm$ is connected, we see that $\pi \boxtimes \eta^{-1}$ is $\chi_M$-generic.  Let $H_M: M \rightarrow \cala$ be defined by
\[ q_F^{\langle \nu, H_M(m)\rangle} = |\nu(m)|_F,\; m \in M, \nu \in \cala^*.\]
 For each $s \in \C$, let
\[I(s, \pi \boxtimes \eta^{-1}) := \ind_P^G ((\pi \boxtimes \eta^{-1}) \boxtimes q^{\langle s\tilde{\alpha}_n+\rho, H_M(-)\rangle})\]
denote the normalized parabolically induced representation. Then this induced representation is $\chi$-generic  and let $\lambda_\chi(s, \pi \boxtimes \eta^{-1})$ be the Whittaker functional as in Proposition 3.1 of \cite{Sha81}.  Also let $w_0(\pi \boxtimes \eta^{-1})$ denote the representation of $M$ given by $w_0(\pi \boxtimes \eta^{-1})(m) = (\pi \boxtimes \eta^{-1})(\tw_0 m \tw_0^{-1})$  and let
\[A(s, \pi \boxtimes \eta^{-1}, \tw_0): I(s, \pi \boxtimes \eta^{-1}) \rightarrow  I(-s, w_0(\pi \boxtimes \eta^{-1}))\]
denote the standard intertwining operator as in \cite{Sha81}. The local coefficient $C_\chi(s, \pi \boxtimes \eta^{-1}, \tw_0)$ arises from the uniqueness of the Whittaker models of the induced representations, that is, it satisfies the equation
\begin{equation} \label{LScoefficient}
\lambda_\chi(s, \pi \boxtimes \eta^{-1}) = C_\chi(s, \pi \boxtimes \eta^{-1}, \tw_0)\lambda_\chi(-s, w_0(\pi \boxtimes \eta^{-1})) \circ A(s, \pi \boxtimes \eta^{-1}, \tw_0).
\end{equation}
This local coefficient is a meromorphic function of $s$ and, in fact, is a rational function of $q^{-s}$. It gives rise to the theory of twisted symmetric and exterior square $L$- and $\epsilon$-factors in the Langlands-Shahidi method, to be studied in the next section. 

 Unlike the local coefficient and the theory of $\gamma$-factors in the Langlands-Shahidi method that is available only for generic representations, the Plancherel measure is defined for any irreducible, admissble representation of $M$.  The Plancherel measure can be viewed as a coarser invariant than the $\gamma$-factor, in the sense that when the representation is additionally generic, Shahidi proved in characteristic 0 that the Plancherel measure can be expressed as a certain product of local $\gamma$-factors. The quantity played an important role in the Gan-Takeda characterization \cite{GT11} of the local Langlands correspondence for non-generic supercuspidal representations of $\GSp_4(F)$ where $F$ is a non-archimedean local field of characteristic 0 and consequently also appeared as part of the corresponding result over local function fields in \cite{Gan13}. 

Let us briefly recall the definition of the Plancherel measure. Let
\[\gamma_{w_0}(G/P) = \int_{\bar{N}} q_F^{\langle 2\rho, H_M(\bar{n}) \rangle}d\bar{n},\]
where $d\bar{n}$ is the Haar measure fixed in \ref{Haar}, $H_{M}$ is defined on $M$ as above and extended to a function on $G$ by letting $H_M(mnk) = H_M(m), m \in M, n \in N$ and $k \in \bfg(\calo)$ (cf. Section 2 of \cite{Sha90}). 
 There exists a constant $\mu(s, \pi \boxtimes \eta^{-1}, w_0)$ which depends only on $s$,  the class of $\sigma$ and on $w_0$, but not on the choice $\tw_0$, such that
\begin{equation}
 A(s, \pi \boxtimes \eta^{-1}, \tw_0) \circ A(-s, w_0(\pi \boxtimes \eta^{-1}), \tw_0^{-1}) = \mu(s, \pi \boxtimes \eta^{-1}, w_0)^{-1}\gamma_{w_0}(G/P)^2 
\end{equation}
The scalar $\mu(s, \pi \boxtimes \eta^{-1}, w_0)$ is a meromorphic function of $s$ and is called the Plancherel measure associated to $s, \pi \boxtimes \eta^{-1}$ and $w_0$.

\section{Twisted exterior square and symmetric square factors}\label{TwistedfactorsFunctionfields}

We now work with the Langlands-Shahidi method for generalized spin groups. In particular, we study the case of a Siegel parabolic to obtain twisted exterior and symmetric square $\gamma$-factors. In characteristic zero, the theory is due to Shahidi \cite{Sha90,Sha97}. In Theorem~\ref{gammathm} below, which addresses the characteristic $p$ case, we extend the results of \cite{HL11} on exterior and symmetric square $\gamma$-factors. The existence part is possible via the results of \cite{Lom09} over function fields and the uniqueness argument is due to Henniart-Lomel\'i, thanks to the local-global technique for ${\rm GL}_n$ found in \cite{HL11,HL13}.

\subsection{Theory of $\gamma$-factors}Let $\mathscr{L}$ be the class whose objects are quadruples $(F,\pi,\eta,\psi)$ consisting of: a non-archimedean local field $F$, a smooth irreducible representation $\pi$ of ${\rm GL}_n(F)$; a multiplicative character $\eta$ of $F^\times$; and, a non-trivial additive character $\psi$ of $F$. We call a quadruple $(F,\pi,\eta,\psi) \in \mathscr{L}$ generic (resp. supercuspidal, tempered, spherical) when the representation $\pi$ of ${\rm GL}_n(F)$ is generic (resp. supercuspidal, tempered, spherical); we say $(F,\pi,\eta,\psi)$ is of degree $n$. When $F$ is of characteristic $p>0$, we write $\mathscr{L}(p)$ to denote the subclass of $\mathscr{L}$ consisting of quadruples $(F,\pi,\eta,\psi)$ with ${\rm char}(F) = p$.

Let $\mathscr{LS}$ be the class of quintuples $(k,\pi,\eta,\psi,S)$ consisting of: a global field $k$; a cuspidal automorphic representation $\pi = \otimes \pi_v$ of ${\rm GL}_n(\mathbb{A}_k)$; a Gr\"ossencharakter $\eta = \otimes \eta_v$ of $\mathbb{A}_k^\times$; a non-trivial character $\psi = \otimes \psi_v$ of $\mathbb{A}_k/k$; and, a finite set of places of $k$ such that $\pi_v$, $\eta_v$ and $\psi_v$ are unramified for $v \notin S$. We focus on the case of a function field $k$ with case ${\rm char}(k) = p$, where we write $\mathscr{LS}(p)$.

\begin{theorem} \label{gammathm}Let $r_0$ be either ${\rm Sym}^2$ or $r = \wedge^2$. There {\bf exists} a rule $\gamma$ on $\mathscr{L}(p)$ that assigns a rational function $\gamma(s,\pi,r_0 \boxtimes \eta,\psi) \in \mathbb{C}(q_F^{-s})$ to every $(F,\pi,\eta,\psi) \in \mathscr{L}(p)$, which is {\bf uniquely} characterized by the following properties:

\begin{enumerate}
\item[(i)] (Naturality). Let $(F,\pi,\eta,\psi) \in \mathscr{L}(p)$, and let $\iota: F' \rightarrow F$ be an isomorphism of locally compact fields. Let $(F',\pi',\eta',\psi') \in \mathscr{L}(p)$ be the quadruple obtained from $(F,\pi,\eta,\psi)$ via $\iota$. Then
\begin{equation*}
   \gamma(s,\pi,r_0 \boxtimes \eta,\psi) = \gamma(s,\pi',r_0 \boxtimes \eta',\psi').
\end{equation*}

\item[(ii)] (Isomorphism). Let $(F,\pi,\eta,\psi) \in \mathscr{L}(p)$, and let $(F,\pi',\eta,\psi) \in \mathscr{L}(p)$ be such that $\pi' \simeq \pi$. Then
\begin{equation*}
   \gamma(s,\pi',r_0 \boxtimes \eta,\psi) = \gamma(s,\pi,r_0 \boxtimes \eta,\psi).
\end{equation*}

\item[(iii)](Relation with Artin factors). Let $(F,\pi,\eta,\psi) \in \mathscr{L}(p)$ be spherical. Let $\sigma$ be the Weil-Deligne representation corresponding to $\pi$ via the local Langlands correspondence; and, identify $\eta$ with a character of $\mathcal{W}_F$ via local class field theory. Then
\begin{equation*}
   \gamma(s,\pi, r_0 \boxtimes \eta, \psi) = \varepsilon(s, (r_0 \circ \sigma) \otimes \eta, \psi) 
    \frac{L(1-s, (r_0 \circ \sigma^\vee) \otimes \eta^{-1})}{L(s, (r_0 \circ \sigma) \otimes \eta)}.
\end{equation*}

\item[(iv)] (Dependence on $\psi$). Let $(F,\pi,\eta,\psi) \in \mathscr{L}(p)$ be of degree $n$. Given $a \in F^\times$, let $\psi^a: F \rightarrow \mathbb{C}^\times$, $x \mapsto \psi(ax)$. Then
\begin{equation*}
   \gamma(s,\pi,r_0 \boxtimes \eta,\psi^a) = \omega_\pi(a)^{n \pm 1} \eta(a)^{\frac{n \pm 1}{2}} \left| a \right|_F^{\frac{n(n \pm 1)}{2} (s-\frac{1}{2})} \gamma(s,\pi,r_0 \boxtimes \eta,\psi),
\end{equation*}
where $\omega_\pi$ is the central character of $\pi$ (we use $+$ for ${\rm Sym}^2 \boxtimes \eta$ and $-$ for $\wedge^2 \boxtimes \eta$).

\item[(v)] (Multiplicativity). For $i=1,\ldots,d$, let $(F,\pi_i,\eta,\psi) \in \mathscr{L}(p)$. Let $\pi$ be an irreducible subquotient of the representation of ${\rm GL}_n(F)$ parabolically induced from $\pi_1 \boxtimes \cdots \boxtimes \pi_d$. Assume that either: $\rm (a)$ $\pi$ is generic, or $\rm (b)$ all of the $\pi_i's$ are quasi-tempered and $\pi$ is the Langlands quotient of the parabolically induced representation. Then
\begin{equation*}
   \gamma(s,\pi,r_0 \boxtimes \eta,\psi) = 
   \prod_{i=1}^{d} \gamma(s, \pi_i, r_0 \boxtimes \eta, \psi)
   \prod_{i<j} \gamma(s,\pi_i \times (\pi_j \cdot \eta),\psi),
\end{equation*}
where each $\gamma(s,\pi_i \times (\pi_j \cdot \eta),\psi)$ is a Rankin-Selberg $\gamma$-factor.

\item[(vi)] (Twists by unramified characters). Let $(F,\pi,\eta,\psi) \in \mathscr{L}(p)$, then
\begin{equation*}
   \gamma(s + s_0,\pi,r_0 \boxtimes \eta,\psi) = \gamma(s,\left| \det(\cdot) \right|_F^{\frac{s_0}{2}} \pi,r_0 \boxtimes \eta,\psi).
\end{equation*}

\item[(vii)] (Global functional equation). Let $(k,\pi,\eta,\psi,S) \in \mathscr{LS}(p)$. Then
\begin{equation*}
L^S(s,\pi,r_0 \boxtimes \eta) = \prod_{v \in S} \gamma(s,\pi_v, r_0 \boxtimes \eta_v, \psi_v) \, L^S(1-s,\pi^\vee,r_0 \boxtimes \eta^{-1}).
\end{equation*}

\end{enumerate}
\end{theorem}
\begin{Remark} \label{partialL} Given $(k,\pi,\eta,\psi,S) \in \mathscr{LS}(p)$ and $v \notin S$, the representation $\pi_v$ is of the form
\begin{equation*}
   \pi_v \hookrightarrow {\rm ind}_{B_n}^{{\rm GL}_n(k_v)}(\chi_{1,v} \boxtimes \cdots \boxtimes \chi_{n,v}),
\end{equation*}
where $\chi_{1,v}$, \ldots, $\chi_{n,v}$ are unramified characters of $k_v^\times$. The corresponding spherical local $L$-functions for $v \notin S$ are defined by
\begin{equation*}
   L(s,\pi_v,{\rm Sym}^2 \boxtimes \eta_v) = \prod_{i \leq j} \dfrac{1}{1 - \chi_{i,v}(\varpi_v) \chi_{j,v}(\varpi_v) \eta_v(\varpi_v) q_v^{-s}}
\end{equation*}
and
\begin{equation*}
   L(s,\pi_v,\wedge^2 \boxtimes \eta_v) = \prod_{i < j} \dfrac{1}{1 - \chi_{i,v}(\varpi_v) \chi_{j,v}(\varpi_v) \eta_v(\varpi_v) q_v^{-s}}.
\end{equation*}
The partial $L$-functions appearing in the global functional equation are then given by
\begin{equation*}
   L^S(s,\pi,r_0 \boxtimes \eta) = \prod_{v \notin S} L(s,\pi_v, r_0 \boxtimes \eta_v).
\end{equation*}
\end{Remark}

\subsection{$L$-functions and $\varepsilon$-factors}\label{LEpsilonFactors}

We now obtain a system of $L$-functions and $\varepsilon$-factors from the system of $\gamma$-factors on $\mathscr{L}$ defined in \S~3. First for tempered $(F,\pi,\eta,\psi) \in \mathscr{L}$ and then in general via the Langlands parametrization and analytic continuation. We focus on the case of characteristic $p$; all of the corresponding results when ${\rm char}(F) = 0$ are due to Shahidi \cite{Sha90}. We begin with the local functional equation of $\gamma$-factors:

\begin{Corollary} (Local functional equation). Let $(F,\pi,\eta,\psi) \in \mathscr{L}(p)$, then
   \begin{equation*}
      \gamma(s,\pi,r_0 \boxtimes \eta,\psi) \, \gamma(1-s,\pi^\vee,r_0 \boxtimes \eta^{-1},\overline{\psi}) = 1.
   \end{equation*}
   \end{Corollary}
\begin{proof}We can actually supply two independent proofs. The former uses the local-global technique employed in the proof of Theorem~3.1, see the proof of Corrollary~2.4 of \cite{HL13}. And the latter utilizes the transfer between local fields to import the local functional equation from the characteristic zero case to the characteristic $p$ case, see \cite{Gan13}. 
\end{proof}
We next recall the definition of $L$-functions.

\begin{itemize}
   \item[(viii)] \emph{Let $(F,\pi,\eta,\psi) \in \mathscr{L}(p)$ be tempered. Let $P_{\pi,\eta}(Z)$ be the polynomial in $Z=q_F^{-s}$ that has the same zeros as the rational function $\gamma(s,\pi,r_0 \boxtimes \eta,\psi) \in \mathbb{C}(Z)$ and has $P_{\pi,\eta}(0) = 1$. Then
   \begin{equation*}
      L(s,\pi,r_0 \boxtimes \eta) = P_{\pi,\eta}(Z)^{-1}.
   \end{equation*}
   }
\end{itemize}

Notice that twisted local $L$-functions are independent of the additive character $\psi$. Proposition~7.2 of \cite{Sha90} applies to the Siegel Levi case for ${\rm GSpin}$ groups, since $r = r_1$ is irreducible. Note that tempered representations of ${\rm GL}_n(F)$ are generic. We thus have:

\begin{itemize}
   \item[(ix)] (Tempered representations). \emph{Let $(F,\pi,\eta,\psi) \in \mathscr{L}$(p) be tempered, then $L(s,\pi,r_0 \boxtimes \eta)$ is holomorphic for $\Re(s) >0$.
   }
\end{itemize}

The next property gives the definition of tempered $\varepsilon$-factors.

\begin{itemize}
   \item[(x)] (Root numbers). \emph{Let $(F,\pi,\eta,\psi) \in \mathscr{L}(p)$ be tempered. Then there exists a monomial $\varepsilon(s,\pi,r_0 \boxtimes \eta,\psi) \in \mathbb{C}(Z)$, such that
   \begin{equation*}
      \gamma(s,\pi,r_0 \boxtimes \eta,\psi) \, L(s,\pi,r_0 \boxtimes \eta) = \varepsilon(s,\pi,r_0 \boxtimes \eta,\psi) \, L(1-s,\pi^\vee,r_0 \boxtimes \eta^{-1}).
   \end{equation*}
   }
\end{itemize}

We now turn towards twisted local $L$-functions and $\varepsilon$-factors for not necessarily generic $(F,\pi,\eta,\psi) \in \mathscr{L}(p)$. Notice that every irreducible smooth representation $\pi$ of ${\rm GL}_n(F)$ is the Langlands quotient of a representation that is parabolically induced from quasi-tempered representations.

\begin{itemize}
   \item[(xi)] (Langlands quotient). \emph{For $i=1,\ldots,d$, let $(\pi_{i,0},\eta,\psi) \in \mathscr{L}(p)$ be tempered. Let $\tau$ be the representation of ${\rm GL}_n(F)$ parabolically induced from $\pi_1 \boxtimes \cdots \boxtimes \pi_d$, where $\pi_i = \left| \cdot \right|_F^{s_i} \pi_{i,0}$ and the $s_i$'s are real numbers with $s_1 > \cdots > s_d$. Assume that $\pi$ is the unique Langlands quotient of $\tau$. Then
\begin{align*}
   L(s,\pi,r_0 \boxtimes \eta) &= 
   \prod_{i=1}^{d} L(s + 2s_i, \pi_{i,0}, r_0 \boxtimes \eta)
   \prod_{i<j} L(s + s_i + s_j,\pi_{i,0} \times (\pi_{j,0} \cdot \eta)), \\
   \varepsilon(s,\pi,r_0 \boxtimes \eta,\psi) &= 
   \prod_{i=1}^{d} \varepsilon(s + 2 s_i, \pi_{i,0}, r_0 \boxtimes \eta, \psi)
   \prod_{i<j} \varepsilon(s + s_i + s_j,\pi_{i,0} \times (\pi_{j,0} \cdot \eta),\psi),
\end{align*}
where each $\varepsilon(s,\pi_{i,0} \times (\pi_{j,0} \cdot \eta),\psi)$ and $L(s,\pi_{i,0} \times (\pi_{j,0} \cdot \eta))$ are Rankin-Selberg $L$-functions and root numbers.
   }
\end{itemize}

\section{Proof of Theorem~\ref{gammathm}}

\subsection{The generic case} Recall that ${\bf G} = {\rm GSpin}_m$ with $m = 2n+1$ or $2n$. Let $\bf M$ be the Siegel Levi subgroup of $\bf G$ corresponding to $\Delta - \left\{ \alpha_n \right\}$. We first deal with the case of a generic quadruple $(F,\pi,\eta,\psi) \in \mathscr{L}(p)$. The representation $\pi$ lifts to a representation $\tau = \pi \boxtimes \eta^{-1}$ of $M \hookrightarrow G$. When $(F,\pi,\eta,\psi) \in \mathscr{L}(p)$ is spherical, we have
\begin{equation} \label{sphericalpi}
   \pi \hookrightarrow {\rm ind}_{B_n}^{{\rm GL}_n(F)}(\chi_1 \boxtimes \cdots \boxtimes \chi_n),
\end{equation}
with the $\chi_1$, \ldots, $\chi_n$ unramified characters of $F^\times$. Up to conjugacy in ${\rm GL}_n(\mathbb{C})$, the Satake classification gives
\begin{equation*}
   \pi \longleftrightarrow {\rm diag}(\chi_1(\varpi_F), \ldots, \chi_n(\varpi_F)).
\end{equation*}
Then, the Langlands parameter for the representation $\tau$ corresponds to the diagonal matrix
\begin{equation*}
   {\rm diag}(\chi_1(\varpi_F), \ldots, \chi_n(\varpi_F),\chi_n^{-1}(\varpi_F)\eta(\varpi_F), \ldots, \chi_1^{-1}(\varpi_F)\eta(\varpi_F))
\end{equation*}
of ${}^LG^0 = {\rm GSp}_{2n}(\mathbb{C})$, if $m = 2n+1$, or ${}^LG^0 = {\rm GSO}_{2n}(\mathbb{C})$, if $m = 2n$. Let ${\rm sim}$ denote the similitude character defining ${}^LG^0$. The adjoint action $r$ of ${}^LM$ on ${}^L\mathfrak{n}$ is then $r = (r_0 \circ \rho_n) \boxtimes {\rm sim}^{-1}$ as in \eqref{adjointr}. The Langlands-Shahidi local coefficient \cite{Sha81} leads us to define
\begin{equation}\label{gammafactordefinition}
   \gamma(s,\tau,(r_0 \circ \rho_n) \boxtimes {\rm sim}^{-1},\psi) \mathrel{\mathop:}= C_\chi(s,\tau,w_0).
\end{equation}
The results on the local coefficient of \cite{Lom09} over function fields, allow us to prove all of the properties (i)--(vii) of the Theorem are satisfied for the LS twisted $\gamma$-factors $\gamma(s,\tau,(r_0 \circ \rho_n) \boxtimes {\rm sim}^{-1},\psi)$ in the generic case. Indeed, notice that $r = r_1$ is irreducible and that given a quintuple $(k,\pi,\eta,\psi,S) \in \mathscr{LS}(p)$ we then obtain a representation $\tau$ of ${\bf M}(\mathbb{A}_k)$ from $\pi \boxtimes \eta^{-1}$. We then have that the crude functional equation, Theorem~5.14 of [op. cit.], reads
\begin{equation*}
   L^S(s,\tau,(r_0 \circ \rho_n) \boxtimes {\rm sim}^{-1}) = \prod_{v \in S} \gamma(s,\tau,(r_0 \circ \rho_n) \boxtimes {\rm sim}^{-1},\psi) 
   \, L^S(1-s,\tau,(r_0 \circ \rho_n^\vee) \boxtimes {\rm sim}),
\end{equation*}
giving us Property~(vii) with the above definition of $\gamma$-factors. In Property~(v), we explicitly state the multiplicativity property of the local coefficient found in \S~2 of [op. cit.]. For the relation with Artin $\gamma$-factors, let $(F,\pi,\eta,\psi) \in \mathscr{L}(p)$ spherical of degree $n$. Multiplicativity helps us study the case of a principal series representation, where $\pi$ is given by \eqref{sphericalpi}, with the $\chi_i$'s not necessarily unramified. In this case we obtain
\begin{align} \label{psgamma}
   \gamma (s,\tau,(r_0 \circ \rho_n) \boxtimes {\rm sim}^{-1},\psi)
   	=	\prod_{i=1}^{n} \gamma(s, \tau_i, (r_0 \circ \rho_{n_i}) \boxtimes {\rm sim}^{-1},\psi)
   		\prod_{i<j} \gamma(s,\chi_i \chi_j \eta,\psi),
\end{align}
where the $\gamma$-factors $\gamma(s,\chi_i \chi_j \eta,\psi)$ are abelian $\gamma$-factors of Tate's thesis. Thus reducing the proof of Property~(iii) to the case of $n=1$. The approach taken in \cite{Lom12b} can be used to complete the proof of this and the remaining properties. For example, Section~3 of [\emph{op. cit.}] for generalized spin groups gives the following equalities for degree $1$ quadruples $(\chi,\eta,\psi) \in \mathscr{L}(p)$ in terms of abelian $\gamma$-factors: 
\begin{equation*}\label{Lquot}
   \gamma(s,\tau,(r_0 \circ \rho_1) \boxtimes {\rm sim}^{-1},\psi) = \left\{ \begin{array}{ll} 	\gamma(s,\chi^2 \eta,\psi) & \text{if } {\bf G} = {\rm GSpin}_{3} \\
												1 & \text{if } {\bf G} = {\rm GSpin}_{2}
					\end{array} \right. .
\end{equation*}

\subsection{The general case} Next, we address the case of a not necessarily generic quadruple $(F,\pi,\eta,\psi) \in \mathscr{L}(p)$. In this generality, we have that $\pi$ is the unique Langlands quotient ${\rm J}(\xi)$ of a representation $\xi$ of the form
\begin{equation}
   \xi = {\rm ind}_{P_n}^{{\rm GL}_n(F)} (\pi_1 \boxtimes \cdots \boxtimes \pi_d).
\end{equation} 
Here, $\bfp_n = \GL_{n_1} \times \cdots \times \GL_{n_d}$ and each $\pi_i$, $1 \leq i \leq d$, is a representation of ${\rm GL}_{n_i}(F)$ of the form $\left| \det(\cdot) \right|_F^{s_i} \pi_{i,0}$ with $s_1 > \cdots > s_d$ and each $\pi_{i,0}$ tempered. We then have
\begin{align} \label{gammadef}
   \gamma &(s,\tau,(r_0 \circ \rho_n) \boxtimes {\rm sim}^{-1},\psi) \\ \nonumber
   	&=	\prod_{i=1}^{d} \gamma(s + 2s_i, \tau_i, (r_0 \circ \rho_{n_i}) \boxtimes {\rm sim}^{-1},\psi)
   		\prod_{i<j} \gamma(s + s_i + s_j,\pi_i  \times (\pi_j \cdot \eta),\psi).
\end{align}
This addresses case~(b) of Property~(v). Finally, the rule $\gamma$ on $\mathscr{L}(p)$ defined by $\gamma(s,\pi,r_0 \boxtimes \eta,\psi) \mathrel{\mathop:}= \gamma (s,\tau,(r_0 \circ \rho_n) \boxtimes {\rm sim}^{-1},\psi)$ satisfies properties (i)--(vii).

For uniqueness, let $\gamma$ be a rule on $\mathscr{L}(p)$ satisfying the hypothesis of the theorem. Given $(F,\pi,\eta,\psi) \in \mathscr{L}(p)$, Properties~(v) and (vi) give that $\gamma(s,\pi,r_0 \boxtimes \eta,\psi)$ depends only on the supercuspidal content of $\pi$ (see Remark~2.3 of \cite{HL13}). Fix a supercuspidal quadruple $(F,\pi,\eta,\psi)$ of $\mathscr{L}(p)$. Theorem~3.1 of \cite{HL13} gives a global field $k$, a cuspidal automorphic representation $\xi = \otimes \, \xi_v$ of ${\rm GL}_n(\mathbb{A}_k)$ such that: $\xi_0 \simeq \pi$ and $\xi_v$, $v \neq 0$ is a subquotient of a principal series representation. Applying the Grundwald-Wang theorem of class field theory gives us a Gr\"ossencharakter $\nu = \otimes \nu_v$ such that $\nu_0 \simeq \eta$. We then have a quintuple $(k,\xi,\nu,\phi,S) \in \mathscr{LS}(p)$, by taking an arbitrary non-trivial character $\phi = \otimes \, \phi_v$ of $\mathbb{A}_k/k$. The above discussion leading towards \eqref{psgamma} shows that Property~(v.a) reduces the study of $\gamma(s,\xi_v,r_0 \boxtimes \nu_v,\phi_v)$ for $v \neq 0$ to abelian $\gamma$-factors, which are uniquely determined. Also, partial $L$-functions are uniquely determined by the formulas of Remark~\ref{partialL}. The functional equation for $(k,\xi,\nu,\phi,S) \in \mathscr{LS}(p)$ reads
\begin{equation*}
L^S(s,\xi,r_0 \boxtimes \nu) = \gamma(s,\xi_0,r_0 \boxtimes \nu_0,\phi_0) \prod_{v \in S - \left\{ 0 \right\}} \gamma(s,\xi_v, r_0 \boxtimes \nu_v, \phi_v) \, L^S(1-s,\xi^\vee,r_0 \boxtimes \nu^{-1}).
\end{equation*}
Properties~(i) and (ii) now give that $\gamma(s,\xi_0,r_0 \boxtimes \nu_0,\phi_0) = \gamma(s,\pi,r_0 \boxtimes \eta,\phi_0)$. Finally, by Property~(iv) we can uniquely determine $\gamma(s,\pi,r_0 \boxtimes \eta,\psi)$ from the uniquely determined $\gamma(s,\pi,r_0 \boxtimes \eta,\phi_0)$. \qed

\section{Compatibililty of the twisted local factors with the local Langlands correspondence and stability}
In this section, we prove the compatibility of the twisted local factors with the local Langlands correspondence in positive characteristic of \cite{LRS93} and deduce an important stability property of the $\gamma$-factors under twists by highly ramified characters.
 \subsection{ Equality with Artin factors} Let $\mathscr{G}$ be the class consisting of quadruples $(F,\sigma,\eta,\psi)$, where: $F$ is a non-archimedean local field; $\sigma$ is an $n$-dimensional ${\rm Frob}$-semisimple Weil-Deligne representation; $\eta$ is a character of $\mathcal{W}_F$; and, $\psi$ is a non-trivial additive character of $F$. Given a prime number $p$, we denote $\mathscr{G}(p)$ the subclass of $\mathscr{G}$ whose objects are quadruples $(F,\sigma,\eta,\psi)$ with ${\rm char}(F) = p$.

We have a rule $\gamma$ on $\mathscr{G}$ which assigns a rational function $\gamma(s,\sigma,r_0 \otimes \eta,\psi)$ to every quadruple $(F,\sigma,\eta,\psi) \in \mathscr{G}$, defined by
\begin{equation} \label{galgamma}
   \gamma(s,r_0(\sigma) \otimes \eta,\psi) \mathrel{\mathop:}= 
   \varepsilon(s,(r_0 \circ \sigma) \otimes \eta,\psi) \, \dfrac{L(1-s,(r_0 \circ \sigma^\vee) \otimes \eta^{-1})}{L(s,(r_0 \circ \sigma) \otimes \eta)} \,,
\end{equation}
where the Artin $L$-functions and root numbers on the right half side are those defined by Deligne and Langlands \cite{t79}.

Consider the morphism
\begin{equation*}
\begin{CD}
   \mathscr{L} @>{\rm LLC}>> 	\mathscr{G}
\end{CD}
\end{equation*}
taking $(F,\pi,\eta,\psi)$ to $(F,\sigma,\eta,\psi)$ via the local Langlands correspondence. More precisely, we let $\sigma = \sigma(\pi)$ be the Weil-Deligne representation corresponding to $\pi$ under the local Langlands correspondence \cite{HT01,Hen00,LRS93}; we identify the character $\eta$ of $F^\times$ with a character of $\mathcal{W}_F$ via local class field theory; and $\psi$ remains unchanged.

\begin{theorem}\label{artintwisted}The morphism
\begin{equation*}
\begin{CD}
   \mathscr{L}(p) @>{\rm LLC}>> 	\mathscr{G}(p)
\end{CD}
\end{equation*}
preserves twisted exterior and symmetric square local factors. More precisely, let ${\rm LLC}: (F,\pi,\eta,\psi) \mapsto (F,\sigma,\eta,\psi)$, then
\begin{align*}
   \gamma(s,\pi,r_0 \boxtimes \eta,\psi) &= \gamma(s,(r_0 \circ \sigma) \otimes \eta,\psi), \\
   L(s,\pi,r_0 \boxtimes \eta) &= L(s,(r_0 \circ \sigma) \otimes \eta) \\
   \varepsilon(s,\pi,r_0 \boxtimes \eta,\psi) &= \varepsilon(s,(r_0 \circ \sigma) \otimes \eta,\psi).
\end{align*}
\end{theorem}
\begin{proof}
 The morphism ${\rm LLC}$ allows us to define a rule $\gamma^{\rm Gal}$ on $\mathscr{L}(p)$ by setting
\begin{equation*}
   \gamma^{\rm Gal}(s,\pi,r_0 \boxtimes \eta,\psi) \mathrel{\mathop:}= \gamma(s,(r_0 \circ \sigma(\pi)) \otimes \eta,\psi),
\end{equation*}
for every $(F,\pi,\eta,\psi) \in \mathscr{L}(p)$, where the $\gamma$-factors on the right are the Artin factors defined by equation~\eqref{galgamma}. The approach is to apply the uniqueness part of Theorem~3.1, in order to conclude that
\begin{equation*}
   \gamma(s,\pi,r_0 \boxtimes \eta,\psi) = \gamma^{\rm Gal}(s,\pi,r_0 \boxtimes \eta,\psi).
\end{equation*}
First, notice that the $\gamma$-factor $\gamma(s,(r_0 \circ \sigma(\pi)) \boxtimes \eta,\psi)$ depends only on the underlying Weil representation of the Weil-Deligne representation $\sigma = \sigma(\pi)$, in addition to the character $\eta$. Hence, we are reduced to representations of the Weil group. Artin factors depend only on the isomorphism class of $\sigma$ and are compatible with class field theory, that is $\gamma^{\rm Gal}$ satisfies Property~(ii) of Theorem~3.1. Properties~(i) and (iii) are immediate. Formulas~3.4.4 and 3.4.5 of \cite{t79} gives Properties~(iv) and (vi). Multiplicativity follows from the property
\begin{equation*}
   r_0(\sigma \oplus \tau) \otimes \eta = (r_0(\sigma) \otimes \eta) \oplus (r_0(\tau) \otimes \eta) \oplus (\sigma \otimes (\tau \otimes \eta))
\end{equation*}   
of twisted exterior and symmetric square representations, which gives
\begin{equation*}
   \gamma(s,r_0(\sigma \oplus \tau) \otimes \eta,\psi) = \gamma(s,r_0(\sigma) \otimes \eta,\psi) \gamma(s,r_0(\tau) \otimes \eta,\psi) \gamma(s,\sigma \otimes (\tau\otimes \eta),\psi).
\end{equation*}
Finally Artin factors satisfy a global functional equation. An important point here is that the global Langlands correspondence for ${\rm GL}_n$ over function fields \cite{Laf02} is compatible with the local Langlands correspondence (see \S~15.3 of \cite{LRS93}). Thus properties (i)--(vii) are satisfied.

To conclude, we point out that the local Langlands correspondence for ${\rm GL}_n$ is compatible with the above construction of $L$-functions and $\varepsilon$-factors from the tempered case (see for example \cite{Hen02}). Hence, local $L$-functions and $\varepsilon$-factors are preserved.
\end{proof}
\subsection{Stability of $\gamma$-factors} We now turn towards an interesting question in the Langlands-Shahidi method, that of stability of $\gamma$-factors under twists by highly ramified characters. We note that several important cases are proved in \cite{CKPSS08} under certain geometric conditions for orbital integrals, due to the connection between the local coefficient and Bessel models. However, the Siegel Levi case of ${\rm GSpin}$ groups falls in a different category. In positive characteristic, thanks to the compatibility with ${\rm LLC} : \mathscr{L}(p) \longleftrightarrow \mathscr{G}(p)$, we can give a simple proof:

\begin{proposition} \emph{(Stability)}\label{Stability}. Let $(F,\pi_i,\eta,\psi) \in \mathscr{L}(p)$, $i=1$, $2$, be such that $\omega_{\pi_1} = \omega_{\pi_2}$ and $\eta$ is highly ramified. Then
\begin{equation*}
   \gamma(s,\pi_1,r_0 \boxtimes \eta,\psi) = \gamma(s,\pi_2,r_0 \boxtimes \eta,\psi).
\end{equation*}
\end{proposition} 

\subsection*{Proof} For each $i=1$, $2$, let the map ${\rm LLC}$ take the quadruple $(F,\pi_i,\eta,\psi)$ to $(F,\sigma_i,\eta,\psi)$. From the proof of Theorem \ref{artintwisted}, we have that
\begin{equation*}
   \gamma(s,\pi_i,r_0 \boxtimes \eta,\psi) = \gamma^{\rm Gal}(s,\sigma_i,r_0 \boxtimes \eta,\psi).
\end{equation*}
The main result of \cite{DH81} shows that there is an element $c = c(\eta,\psi) \in F^\times$, such that $\psi^c(x) = \eta(1+x)$ for $x \in \mathfrak{p}^{\left[ k/2 \right]+1}$, and
\begin{equation*}
   \gamma^{\rm Gal}(s,\sigma_i,r_0 \boxtimes \eta,\psi) = \det(r_0 \circ \sigma_i(c))^{-1} \, \gamma(s,\eta,\psi)^{\frac{n\pm1}{2}}.
\end{equation*}
The condition $\omega_{\pi_1} = \omega_{\pi_2}$ gives $\det(r_0 \circ \sigma_1(c)) = \det(r_0 \circ \sigma_2(c))$. \qed

\begin{Remark} The results of \cite{HL11} can be applied directly to give a weaker statement of stability requiring $\eta^2$ to be highly ramified (see \S~4.4 of \cite{HL12}). What we have provided here is the full stability property of twisted exterior and symmetric square $\gamma$-factors.
\end{Remark}
\section{Compatibility of the twisted symmetric and exterior square factors with the Deligne-Kazhdan correspondence}
 In \cite{Gan13}, the first author showed that the Langlands-Shahidi local coefficient is compatible with the Deligne-Kazhdan theory for sufficiently close local fields. The results of \cite{Gan13} hold true for local coefficients associated to generic representations of Levi subgroups of a split reductive group. In this section, we will first briefly review the Deligne-Kazhdan correspondence from \cite{Del84, Kaz86}. Since the twisted $\gamma$-factor on the analytic side is defined via the local coefficient (Equation \ref{gammafactordefinition}), we apply the results of \cite{Gan13}  to deduce the compatibility of the twisted symmetric and exterior square local factors with the Deligne-Kazhdan correspondence. 

Recall that two non-archimedean local fields $F$ and $F'$ are \textit{$m$-close} if $\calo_F/\calp_F^m \cong \calo_{F'}/\calp_{F'}^m$.  Note that given a local field $F'$ of characteristic $p$ and an integer $m \geq 1$, there exists a local field $F$ of characteristic 0 such that $F'$ is $m$-close to $F$. We just have to choose the field $F$ to be ramified enough. 
\subsection{Deligne's theory}\label{Delignetheory}

In \cite{Del84},  Deligne proved that if  $F$ and $F'$ are $m$-close, then
\[ \Gal(\bar{F}/F)/I_F^m \cong \Gal(\bar{F'}/F')/I_{F'}^m, \] where $\bar{F}$ is a separable algebraic closure of $F$, $I_F$ is the inertia subgroup and $I_F^m$ denotes the $m$-th higher ramification subgroup of $I_F$ with upper numbering. This gives a bijection
\begin{align*}
& \text{\{Cont., complex, f.d. representations of $\Gal(\bar{F}/F)$ trivial on $I_F^m$\}}\\
& \overset{\Del_m}\longleftrightarrow  \text{\{Cont., complex, f.d. representations of $\Gal(\bar{F'}/F')$ trivial on $I_{F'}^m$\}}. 
\end{align*}
 Furthermore, he proved that the Artin $L$- and $\epsilon$-factors are preserved under $\Del_m$ (Check for Proposition 3.7.1 of \cite{Del84}). Let us recall this result in the context of the twisted symmetric and exterior square factors. To state it precisely, let us first introduce some notation.
 
Given $m \geq 1$, let $\mathscr{G}^m$ be the subclass of $\mathscr{G}$ whose objects consist of quadruples $(F, \sigma, \eta, \psi)$ where: $F$ is a non-archimedean local field, $\depth(\sigma) < m$, $\depth(\eta) < m$ and $\psi$ is a non-trivial additive character of $F$. Note that $\depth(\sigma)<m$ implies that $\sigma$ is trivial on $I_F^m$. 

\begin{definition}\label{DelAsso}Let $(F', \sigma', \eta', \psi') \in \mathscr{G}^m$. We say that the quadruple $(F, \sigma, \eta, \psi) \in \mathscr{G}$  \textit{is $\Del_m$-associated to $(F', \sigma', \eta', \psi')$} if:
\begin{enumerate}
\item $F'$ is $m$-close to $F$,
\item $\sigma' = \Del_m(\sigma)$,
\item $\eta$ corresponds to $\eta'$ via the isomorphism $F^\times/(1+\calp_F^m) \cong F'^\times/(1+\calp_{F'}^m)$,
\item With $k := \cond(\psi')$, $\psi$ satisfies: 
\begin{align*}
 \cond(\psi) &= k,\\\nonumber
\psi|_{\calp_{F}^{k-m}/\calp_{F}^{k}} &= \psi'|_{\calp_{F'}^{k-m}/\calp_{F'}^{k}}.
\end{align*}
\end{enumerate}
\end{definition}
We then have the following proposition.

\begin{proposition}[Proposition 3.7.1 of \cite{Del84}]\label{Prop3.7.1}Let $m \geq 1$ and let $ r_0 = \Sym^2$ or $\wedge^2$. Let $(F', \sigma', \eta', \psi') \in \mathscr{G}^m$. For each $(F, \sigma, \eta, \psi)$ that is $\Del_m$-associated to $(F', \sigma', \eta', \psi')$, we have
\begin{align*}
 \gamma(s,(r_0 \circ \sigma) \otimes \eta,\psi)  &= \gamma(s,(r_0 \circ \sigma') \otimes \eta',\psi'), \\
    L(s,(r_0 \circ \sigma) \otimes \eta) &= L(s,(r_0 \circ \sigma') \otimes \eta') \\
\varepsilon(s,(r_0 \circ \sigma) \otimes \eta,\psi)  &= \varepsilon(s,(r_0 \circ \sigma') \otimes \eta',\psi').
\end{align*}
\end{proposition}
\subsection{Kazhdan's theory} Let ${\bf G}$ be a split, connected reductive group defined over $\mathbb{Z}$ and let $G := {\bf G}(F)$. For the remainder of this section we use the following notation: For an object $X$ associated to the field $F$, we will use the notation $X'$ to denote the corresponding object over $F'$. 
 In \cite{Kaz86}, Kazhdan proved that  given $m \geq 1$, there exists $l \geq m$ such that if $F$ and $F'$ are $l$-close, then there is an algebra isomorphism $\Kaz_m:\mathscr{H}(G, K_m) \rightarrow \mathscr{H}(G', K_m')$, where $K_m $ is the $m$-th usual congruence subgroup of ${\bf G}(\calo_F)$. (In fact, this result holds with $l=m$. Cf. Corollary 4.15 of \cite{Gan13}).
Hence, when the fields $F$ and $F'$ are sufficiently close, we have a bijection
 \begin{align*}
 &\text{\{Irreducible, admissible representations $(\pi, V)$ of $G$ such that $\pi^{K_m} \neq 0$\}} \\
  &\longleftrightarrow\text{\{Irreducible, admissible representations  $(\pi', V')$ of $G'$ such that $\pi'^{K_m'} \neq 0$\}}.
\end{align*}
A variant of the Kazhdan isomorphism is now available for split reductive groups. More precisely, with $\bfg$ as above, a presentation has been written down for the Hecke algebra $\fH(G, I_{m})$, where $I_{m}$ is the $m$-th Iwahori filtration subgroup of $G$ (cf. \cite{How85} for $\GL_n$ and \cite{Gan13} for any split reductive group).  It has then been shown that when $F$ and $F'$ are $m$-close, 
\begin{equation}\label{variantKazhdan}
\fH(G, I_{m}) \cong \fH(G', I_{m}')
\end{equation} (cf. \cite{Lem01} for $\GL_n$ and \cite{Gan13} for any split reductive group).   This gives rise to a bijection
 \begin{align}\label{IwahoriHeckeIso}
 &\text{\{Irreducible, admissible representations $(\pi, V)$ of $G$ such that $\pi^{I_{m}} \neq 0$\}} \\
  &\overset{\kappa_m}\longleftrightarrow\text{\{Irreducible, admissible representations  $(\pi', V')$ of $G'$ such that $\pi'^{I_{m'}} \neq 0$\}}.\nonumber 
\end{align}

\subsection{Compatibility of twisted symmetric and exterior square factors over close local fields}
In \cite{Gan13}, this variant of the Kazhdan isomorphism (Equation \ref{variantKazhdan}) was used to study the Langlands-Shahidi local coefficients over close local fields.  In this section, we will apply Theorem 6.5 of \cite{Gan13} to deduce the analogue of Proposition \ref{Prop3.7.1} above on the analytic side. 

Given $m \geq 1$, let $\mathscr{L}^m$ be the subclass of $\mathscr{L}$ whose objects consist of quadruples $(F, \pi, \eta, \psi)$ where: $F$ is a non-archimedean local field, $\depth(\pi) < m$, $\depth(\eta) < m$ and $\psi$ is a non-trivial additive character of $F$. Note that by Lemma  8.2 of \cite{Gan13}, $\depth(\pi) < m \implies \pi^{I_{m+1}} \neq 0$,
\begin{definition}\label{KazAsso}
Let $l \geq m+1$ and let $(F', \pi', \eta', \psi') \in \mathscr{L}^m$. We say that the quadruple $(F, \pi, \eta, \psi)$  \textit{is $\Kaz_l$-associated to $(F', \pi', \mu', \psi')$} if:
\begin{enumerate}
\item $F'$ is $l$-close to $F$,
\item $\pi' = \kappa_l(\pi)$,
\item $\eta$ corresponds to $\eta'$ via the isomorphism $F^\times/(1+\calp^l) \cong F'^\times/(1+\calp'^l)$,
\item With $k := \cond(\psi')$, $\psi$ satisfies: 
\begin{align}\label{psicorr}
 \cond(\psi) &= k,\\\nonumber
\psi|_{\calp_{F}^{k-l}/\calp_{F}^{k}} &= \psi'|_{\calp_{F'}^{k-l}/\calp_{F'}^{k}}.
\end{align}
\end{enumerate}
\end{definition}
\begin{Remark}\label{RemarkForLuis}
Note that given any $(F', \pi', \eta', \psi') \in \mathscr{L}(p)$, there exists an integer $l$ and a quadruple $(F, \pi, \eta, \psi)\in \mathscr{L}$ (with $\Char(F) =0$) that is $\Kaz_l$-associated to $(F', \pi', \eta', \psi')$. 
\end{Remark}

We then have the following proposition.
\begin{proposition}\label{Kazhdantwisted}Let $n \geq 2$ and $m \geq 1$. Let $r_0 = \Sym^2$ or $\wedge^2$. Let $(F', \pi', \eta', \psi')  \in \mathscr{L}^m$. There exists $l = l(n,m) \geq m+1$ such that for any $(F, \pi, \eta, \psi)\in \mathscr{L}$ that is $\Kaz_l$-associated to $(F', \pi', \eta', \psi')$, we have
\begin{align*}
L(s, \pi, r_0 \boxtimes \eta) &= L(s, \pi', r_0 \boxtimes \eta')\\
\gamma(s, \pi, r_0 \boxtimes \eta, \psi) &= \gamma(s, \pi', r_0 \boxtimes \eta', \psi')\\
\epsilon(s, \pi, r_0 \boxtimes \eta, \psi) &= \epsilon(s, \pi', r_0 \boxtimes \eta', \psi').
\end{align*} 

\end{proposition}
\begin{proof} This proposition is a consequence of Theorem 6.5 of \cite{Gan13} (and the argument below is similar to Theorem 8.6 of \cite{Gan13}).  Let us first deal with the $\gamma$-factor for a generic quadruple $(F', \pi', \eta', \psi')$. 
Using Property (iv) of Theorem \ref{gammathm}, we may and do  assume that $\cond(\psi) = 0$ (Here note that if $\psi \leftrightarrow \psi'$ as in Equation \ref{psicorr}, then $\psi^a \leftrightarrow \psi'^{a'}$ where $a'$ corresponds to the image of $a$ under the isomorphism $F^\times/(1+\calp^l) \cong F'^\times/(1+\calp'^l)$). For each $l \geq m+1$ and each $(F, \pi, \eta, \psi)\in \mathscr{L}^m$ that is $\Kaz_l$-associated to $(F', \pi', \eta', \psi')$ we have, by Corollary 5.5 of \cite{Gan13}, that the quadruple $(F, \pi, \eta, \psi)$ is also generic. Let $W(\pi, \chi)$ denote the Whittaker model of $\pi$ (where $\chi$ is as in Section \ref{LCPM}). Let $\bfg = \GSpin_{2n+1}$ for $r_0 = \Sym^2$ and $\bfg = \GSpin_{2n}$ for $r_0 = \wedge^2$, and let $I_m$ be the $m$-th Iwahori filtration subgroup of $G$. We show that there exists an $ l = l(n,m) \geq m+1$, depending only on $n$ and $m$, such that 
\begin{itemize}
\item Theorem 8.6 of \cite{Gan13} holds,
\item $V^{I_l \cap \GL_n(F)} \neq 0$,
\item There exists $v_0 \in V^{I_l \cap \GL_n(F)}$ with $W_{v_0}(e) \neq 0$. 
\end{itemize}

Using Theorem 2.3.6.4 of \cite{Yu09}, it is easy to see that
\[\cond(\pi) \leq n^2\depth(\pi)+n^2 \leq n^2m+n^2.\]
Let $v_0$ be the essential vector of the representation $\pi$ as in Theorem 2 of \cite{BH97} and let $W_{v_0} \in W(\pi, \psi)$ be the corresponding Whittaker function. This vector has the property that $W_{v_0}(e) \neq 0$ (Cf. Proposition 1 of \cite{BH97}). Set $m_1 = n^2m+n^2$. Since $\cond(\pi) \leq m_1$, this vector lies in 
\[K_n(m_1) =  \left\{\left(\begin{array}{cc}
A & b \\
c & d \\
\end{array}\right)\in \GL_n(\calo) :A \in \GL_{n-1}(\calo),\, c\equiv 0\, mod \, \calp_{F}^{m_1},\, d \equiv 1\, mod \, \calp_{F}^{m_1}\right\}.\]
Since $I_{m_1} \cap \GL_n(F) \subset K_n(m_1)$, we see that $v_0 \in V^{I_{m_1} \cap \GL_n(F)}$.  Let $l \geq n^2m +n^2+4$ be large enough so that Theorem 8.6 of \cite{Gan13} holds (This theorem is needed for Equation \ref{RSCLF} below).  Then by Theorem 6.5 of \cite{Gan13} and we obtain that for any $(F, \pi, \eta, \psi) $ that  is $\Kaz_l$-associated to $(F', \pi', \eta', \psi')$,
\begin{equation}\label{Equation1}\gamma(s, \pi ,r_0 \boxtimes \eta, \psi) = \gamma(s, \pi',r_0 \boxtimes \eta', \psi').
\end{equation}

This finishes the case of a generic quadruple $(F', \pi', \eta', \psi')$.  

The $\gamma$-factor in the non-generic case is defined using the Langlands classification in Equation \ref{gammadef}. Write $\pi' = J(\xi')$ where 
\begin{equation*}
   \xi' = {\rm ind}_{P_n'}^{{\rm GL}_n(F')} (\pi_1' \boxtimes \cdots \boxtimes \pi_d').
\end{equation*} 
with each $\pi_i'$ quasi-tempered as before. Then $\depth(\pi') = \depth(\pi_1' \boxtimes \cdots \boxtimes \pi_d')$   by Theorem 5.2 of \cite{MP96}. Now, since temperedness, twists by unramified characters and normalized parabolic induction are all compatible with the Deligne-Kazhdan correspondence (Cf. Section 5 of \cite{Gan13}), we in fact obtain that $\pi = J(\xi)$ where
\begin{equation*}
   \xi = {\rm ind}_{P_n}^{{\rm GL}_n(F)} (\pi_1 \boxtimes \cdots \boxtimes \pi_d).
\end{equation*} 
with $\pi_i \leftrightarrow \pi_i'$ via $\kappa_l$ of Equation \ref{IwahoriHeckeIso}. 
Now, by Theorem 8.6 of \cite{Gan13}, we have
\begin{equation}\label{RSCLF}
\gamma(s,\pi_i  \times (\pi_j \cdot \eta),\psi) = \gamma(s,\pi_i'  \times (\pi_j' \cdot \eta'),\psi').
\end{equation}
(Check \cite{Gan12} and \cite{ABPS13} for other proofs of the above equality).
Combining these observations with Equation \ref{Equation1} and Equation \ref{gammadef}, the equality of $\gamma$-factors follows for a not-necessarily generic quadruple $(F', \pi', \eta', \psi')$. 

The $L$-function is first defined for tempered representations via the $\gamma$-factor, and then extended to the general case via the Langlands classification (Cf. Section \ref{LEpsilonFactors}).  Hence the equality of $L$-functions over close local fields follows from the above. Since the root numbers are defined using $L$-functions and $\gamma$-factors, the equality of root numbers over close local fields also follows from the above.\qedhere
\end{proof}

\section{Plancherel measures}
In this section, we will express the Plancherel measure associated to representations of the Siegel Levi subgroups  over local function fields in terms of the $\gamma$-factors described in Theorem \ref{gammathm}, and then deduce their compatibility with the Deligne-Kazhdan correspondence. 
\subsection{Plancherel measures and $\gamma$-factors} We let  $(\bfg, \bfm) = (\GSpin_{m}, \GL_n \times \GL_1)$ with $m = 2n+1$ or $2n$. Let $F$ be a non-archimedean local field. Given $(F,\pi,\eta,\psi) \in \mathscr{L}$ generic, let $\tau$ be the representation of $M$ obtained from $\pi \boxtimes \eta^{-1}$. The definition of Plancherel measure was recalled in Section \ref{LCPM}, where $\rho, w_{0}, \gamma_{w_{0}}(G/P)$ are introduced.

\begin{proposition}\label{Plancherelgamma} Let $(F,\pi,\eta,\psi) \in \mathscr{L}$ and let $\tau$ be obtained from $\pi \boxtimes \eta^{-1}$ as above. Let $r_0 = {\rm Sym}^2$ if $\bfg = {\rm GSpin}_{2n+1}$ and $r_0 = \wedge^2$ if $\bfg = {\rm GSpin}_{2n}$. Then
\begin{align}\label{Plancherelformula}
   \gamma_{w_{0}}(G/P)^{-2}\mu(s, \tau, w_{0}) &={\gamma(s, \pi, r_0 \boxtimes \eta, \psi) \gamma(-s, \pi^\vee, r_0 \boxtimes \eta^{-1}, \bar{\psi})} \\
   & = \gamma(s, (r_0 \circ \sigma) \boxtimes \eta, \psi) \gamma(-s, (r_0 \circ \sigma^\vee) \boxtimes \eta^{-1}, \bar{\psi}). \nonumber 
\end{align}
\end{proposition}
\begin{proof}
With the notation of \S~\ref{LCPM}, we have a composition of intertwining operators
\begin{equation*}
\begin{CD}
   {\rm I}(-s,w_0(\pi \boxtimes \eta^{-1})) @>{\rm A}(-s,w_0(\pi \boxtimes \eta^{-1}),\tilde{w}_0^{-1})>> 
   {\rm I}(s,\pi \boxtimes \eta^{-1}) @>{\rm A}(s,\pi \boxtimes \eta^{-1},\tilde{w}_0)>> {\rm I}(-s,w_0(\pi \boxtimes \eta^{-1})).
\end{CD}
\end{equation*}
Then
\begin{equation*}
   \gamma_{w_0}(G/P)^{-2} \mu(s,\pi \boxtimes \eta^{-1},w_0) = C_\chi(s,\pi \boxtimes \eta^{-1},\tilde{w}_0) C_\chi(-s,w_0(\pi \boxtimes \eta^{-1}),\tilde{w}_0^{-1}),
\end{equation*}
follows by twice applying equation \eqref{LScoefficient}. 

First assume $\pi$ is supercuspidal. Then $\pi = \left| \det(\cdot) \right|_F^{-s_0} \pi_0$, with $\pi_0$ unitary. We have that
\begin{equation*}
   C_\chi(-s,w_0(\pi \boxtimes \eta^{-1}),\tilde{w}_0^{-1}) = C_\chi(-s-2s_0,w_0(\pi_0 \boxtimes \eta^{-1}),\tilde{w}_0^{-1})
\end{equation*}
Then we can apply Proposition 3.1.3 of \cite{Sha81} and Lemma 3.1 of \cite{Sha88} for unitary $\pi_0$, and we see that
\[C_\chi(-s,w_0(\pi \boxtimes \eta^{-1}),\tilde{w}_0^{-1}) = C_{\bar{\chi}}(-s,\pi^\vee \boxtimes \eta,\tilde{w}_0).\]
Using Equation \ref{gammafactordefinition}, we obtain that
\begin{align*}
   \gamma_{w_0}&(G/P)^{-2} \mu(s,\pi \boxtimes \eta^{-1},w_0) \\
      			  &= \gamma(s, \pi \boxtimes \eta^{-1}, (r_0 \circ \rho_n) \boxtimes {\rm sim}^{-1}, {\psi})
			  \gamma(s, \pi^\vee \boxtimes \eta, (r_0 \circ \rho_n) \boxtimes {\rm sim}^{-1}, \bar{\psi})\\
			  & = \gamma(s, \pi, r_0 \boxtimes \eta, {\psi}) \gamma(-s, \pi^\vee, r_0\boxtimes \eta^{-1}, \bar{\psi})
\end{align*}

Next, suppose that
\begin{equation*}
   \pi \hookrightarrow {\rm ind}_{P_n}^{{\rm GL}_n(F)} (\pi_1 \boxtimes \cdots \boxtimes \pi_d),
\end{equation*}
where each $\pi_i$, $1 \leq i \leq j$, is supercuspidal. We then have what is known as Langlands lemma, although we use Shahidi's algorithmic proof (Lemma~2.1.2 of \cite{Sha81}), to decompose Weyl group elements and the corresponding unipotent radicals. This gives
\begin{equation*}
   \tilde{w}_0 = \prod_i \left( \prod_{j < i} \tilde{w}_{i,j} \right) \tilde{w}_i  \text{ and } 
   \delta_{w_0}(G/P) = \prod_i \prod_{j < i} \delta(G_i/P_i)_{w_i} \delta(G_{i,j}/P_{i,j})_{w_{i,j}},
\end{equation*}
and obtain a factorization of the intertwining operators, Theorem~2.1.1 of \cite{Sha81}. This allows us to prove
\begin{align} \label{intdecomp}
   {\rm A}(s,\tau,\tilde{w}_0) {\rm A}(-s,\tilde{w}_0(\tau),\tilde{w}_0^{-1}) & = \prod_i {\rm A}(s,\tau_i,\tilde{w}_i) {\rm A}(-s,\tilde{w}_i(\tau_i),\tilde{w}_i^{-1}) \\
   				& \times \prod_{j<i} {\rm A}(s,\tau_{i,j},\tilde{w}_{i,j}) {\rm A}(-s,\tilde{w}_{i,j}(\tau_{i,j}),\tilde{w}_{i,j}^{-1}) \, . \nonumber
\end{align}
Here, $\tau_i$ is the representation $\pi_i \boxtimes \eta^{-1}$ of $M_i = {\rm GL}_{n_i}(F) \times {\rm GL}_1(F) \hookrightarrow G_i = {\rm GSpin}_{n_i}(F)$. For the representations $\tau_{i,j}$, we let ${\bf M}_{i,j}$ be the Levi subgroup of ${\bf G}_{i,j} = {\rm GSpin}_{n_i + n_j}$ corresponding to $\theta_{i,j} = \Delta \backslash  \left\{ \alpha_{n_i}, \alpha_{n_i+n_j} \right\}$; then $\tau_{i,j}$ is the representation of $M_{i,j}$ lifted from $(\pi_i \boxtimes \pi^\vee_j) \boxtimes \eta^{-1}$ of $({\rm GL}_{n_i}(F) \times {\rm GL}_{n_j}(F)) \times {\rm GL}_1(F) \hookrightarrow G_{i,j} = {\rm GL}_{n_i+n_j}(F) \times {\rm GL}_1(F) \hookrightarrow {\rm GSpin}_{n_i+n_j}(F)$. The Weyl group elements are $w_i = w_{l,G_i} w_{l,M_i}$ and $w_{i,j} = w_{l,G_{i,j}} w_{l,M_{i,j}}$. Then, equation~\eqref{intdecomp} leads to
\begin{equation} \label{mudecomp}
   \delta_{w_0}^{-2}(G/P) \mu(s,\tau,\tilde{w}_0) = \delta_{w_0}^{-2}(G/P) \prod_i \mu(s,\tau_i,\tilde{w}_i) \prod_{j<i} \mu(s,\tau_{i,j},\tilde{w}_{i,j}), 
\end{equation}
where we have equation~\eqref{Plancherelformula} for each $\mu(s,\tau_i,\tilde{w}_i)$. Furthermore, we have that $\mu(s,\tau_{i,j},\tilde{w}_{i,j}) = \mu(s,\pi_i \times (\pi_j \cdot \eta),\tilde{w}_0')$. Plancherel measures for general linear groups and the Langlands-Shahidi local coefficient are explored in \cite{Sha83}, the connection to Rankin-Selberg $\gamma$-factors in characteristic $p$ can be made through the results of \cite{HL13}.

For general not necessarily generic $(F,\pi,\eta,\psi) \in \mathscr{L}$, let $\pi$ obtained as in Property~(xi) of \S~4 from the representation of ${\rm GL}_n(F)$ parabolically induced from $\pi_1 \boxtimes \cdots \boxtimes \pi_d$, with each $\pi_i = \left| \cdot \right|_F^{s_i} \pi_{i,0}$ and $\pi_{i,0}$ is tempered. We first extend equation~\eqref{mudecomp} to the case when $\pi$ is a discrete series representation, and then the tempered case; we omit the details. Plancherel measures in general satisfy
\begin{equation*}
   \delta_{w_0}^{-2}(G/P) \mu(s,\tau,\tilde{w}_0) = \delta_{w_0}^{-2}(G/P) \prod_i \mu(s + 2s_i,\tau_i,\tilde{w}_i) \prod_{j<i} \mu(s + s_i + s_j,\tau_{i,j},\tilde{w}_{i,j}),
\end{equation*}
where the definition is compatible with that of $\gamma$-factors on $\mathscr{L}$. 

The first equality of equation~\eqref{Plancherelformula} is possible in characteristic $0$ thanks to Corollary~3.6 of \cite{Sha90} and is now a consequence of Theorem \ref{gammathm} of this article for local function fields. With respect to the second equality of Equation \ref{Plancherelformula} for local fields of characteristic $0$, we follow the ingenious observation made in the proof of Proposition 9.2 of \cite{GT11}.  More precisely, Henniart \cite{Hen2010} proved that, in characteristic 0, the LLC preserves the twisted exterior and symmetric square $\gamma$-factors up to a root of unity $\alpha$, that is
\begin{align*}
\gamma(s, \pi, r_0 \boxtimes \eta, \psi)  = \alpha \, \gamma(s, (r_0 \circ \sigma) \otimes \eta, \psi).
\end{align*}
Then the key observation is that
\begin{align*}
\gamma(s, \pi, r_0 \boxtimes \eta, \psi) \gamma(-s, \pi^\vee, r_0 \boxtimes \eta^{-1}, \bar{\psi}) &= \frac{\gamma(s, \pi, r_0 \boxtimes \eta, \psi)}{ \gamma(1-s, \pi, r_0 \boxtimes \eta, \psi)} \text{ (Equation 3.10 of \cite{Sha90})}\\
&=\frac{\alpha \, \gamma(s, (r_0 \circ \sigma) \otimes \eta, \psi)  }{\alpha \, \gamma(1-s, (r_0 \circ \sigma) \otimes \eta, \psi) }\\
  &= \gamma(s, (r_0 \circ \sigma) \otimes \eta, \psi) \gamma(-s, (r_0 \circ \sigma^\vee) \otimes \eta^{-1}, \bar{\psi}).
\end{align*}
This proves the second equality of Equation \ref{Plancherelformula}  for local fields of characteristic 0. For local function fields, this second equality is a consequence of Theorem \ref{artintwisted}.
\end{proof}

\subsection{Transfer of Plancherel measures} Having expressed the Plancherel measure in terms of the $\gamma$-factors, we obtain the following corollary on the equality of Plancherel measures over close local fields. More precisely,
\begin{Corollary}\label{Planchereltransfer}
Let $n \geq 2$ and $m \geq 1$. Let $r_0 = \Sym^2$ or $\wedge^2$. Let $(F', \pi', \eta', \psi')  \in \mathscr{L}^m$. There exists $l = l(n,m) \geq m+1$ such that for any $(F, \pi, \eta, \psi)\in \mathscr{L}$ that is $\Kaz_l$-associated to $(F', \pi', \eta', \psi')$, we have
\[\mu(s, \pi \boxtimes \eta^{-1}, w_{0}) = \mu(s, \pi' \boxtimes \eta'^{-1}, w_{0}),\]
\end{Corollary}
\begin{proof} This is a consequence of Proposition~\ref{Plancherelgamma} and Proposition \ref{Kazhdantwisted}. We can also deduce this by combining Proposition~\ref{Plancherelgamma} and Proposition~\ref{Prop3.7.1} of this article with Theorem 8.7 of \cite{Gan13}. A direct proof of this equality, without involving the $\gamma$-factors, can be found in Section 7 of \cite{Gan13}.
\end{proof}

\bibliographystyle{amsplain}
\bibliography{all}
\bigskip
{\sc \Small Radhika Ganapathy, Department of Mathematics, The University of British Columbia, Vancouver}\\
\emph{\Small E-mail address: }\texttt{\Small rganapat@math.ubc.ca}\\
{\sc \Small Luis Lomel\'i, Department of Mathematics, University of Oklahoma, Norman, Oklahoma }\\
\emph{\Small E-mail address: }\texttt{\Small lomeli@math.ou.edu}

\end{document}